\documentclass[11pt]{amsart}

  \usepackage[colorlinks, pagebackref, linkcolor=red, citecolor=blue]{hyperref}
 \usepackage{amsmath,amssymb}
 \usepackage{color}
 \usepackage{mathrsfs}
 \usepackage{graphicx}
 \usepackage{lineno}

 \numberwithin{equation}{section}
 \theoremstyle{plain}
 \newtheorem{theorem}{Theorem}[section]

   \theoremstyle{definition}

 \def\T{ \mathbb T}
 \def\R{ \mathbb R}

 \def\D{{ \mathbb D}}
  
 \def\C{{ \mathbb C}}

 \def\dis{\displaystyle}
 
 \def\Union{\bigcup}

 \def\ov{\overline}

 \def\buildrel#1_#2^#3{\mathrel{\mathop{\kern 0pt#1}\limits_{#2}^{#3}}}
 
\begin{document}

 \title [Pseudohyperbolic geometry]{On a family of pseudohyperbolic disks}


 \author{Raymond Mortini}
  \address{
Universit\'{e} de Lorraine\\
 D\'{e}partement de Math\'{e}matiques et  
Institut \'Elie Cartan de Lorraine,  UMR 7502\\
 Ile du Saulcy\\
 F-57045 Metz, France} 
 \email{raymond.mortini@univ-lorraine.fr}

 \author{Rudolf Rupp}
\address{ Fakult\"at f\"ur Angewandte Mathematik, Physik  und Allgemeinwissenschaften\\
\small TH-N\"urnberg\\
\small Kesslerplatz 12\\
\small D-90489 N\"urnberg, Germany
}
\email  {Rudolf.Rupp@th-nuernberg.de} 

 
 \begin{abstract}
 In der Funktionentheorie der Einheitskreissscheibe $\D$ spielt die hyperbolische Geometrie
 eine zentrale Rolle.  Bekannntlich sind wegen des Lemmas von Schwarz-Pick die holomorphen Isometrien bez\"uglich dieser Geometrie nichts anderes als die konformen Selbstabbildungen von $\D$.
 Wir interessieren uns f\"ur die Schar der Scheiben $D_\rho(x,r)$ mit festem Radius $r$ und 
  $-1<x<1$  bez\"uglich der pseudohyperbolischen Metrik $\rho$ in $\D$, und bestimmen mittels funktionentheoretischer Mitteln explizit deren Einh\"ullende.  
  \end{abstract}

  \maketitle

 \centerline {\small\the\day.\the \month.\the\year} \medskip
 \section{Introduction}

 {\sl Hyperbolic geometry \footnote{The slanted text  stems from \cite{cfkp}}
  was created in the first half of the nineteenth century 
in the midst of attempts to understand Euclid's axiomatic basis for geometry.}
 The mathematicians at that time were mainly driven by the question 
 whether the parallel axiom was redundant or not. It turned out that it was not.
{\sl  Hyperbolic geometry is now one type of non-Euclidean geometry
 that  discards the parallel axiom. 
 Einstein and Minkowski found in non-Euclidean geometry a
 geometric basis for the understanding of physical time and space. These
 negatively curved geometries, of which
hyperbolic non-Euclidean geometry is the prototype, are the generic forms of geometry.
They have profound applications to the study of complex variables, to
the topology of two- and three-dimensional manifolds, to group theory, 
 to physics, and to other disparate fields of mathematics. }  Outside mathematics, hyperbolic tesselations of the unit disk have been rendered  very popular by the artist  M.C. Escher.  
 A nice introduction into hyperbolic geometry is, for example,  given in  the monograph 
 \cite{an} and  in \cite{cfkp}.\\
 
 \begin{figure}[h!]
   \hspace{3.7cm}
   \scalebox{.40} {\includegraphics{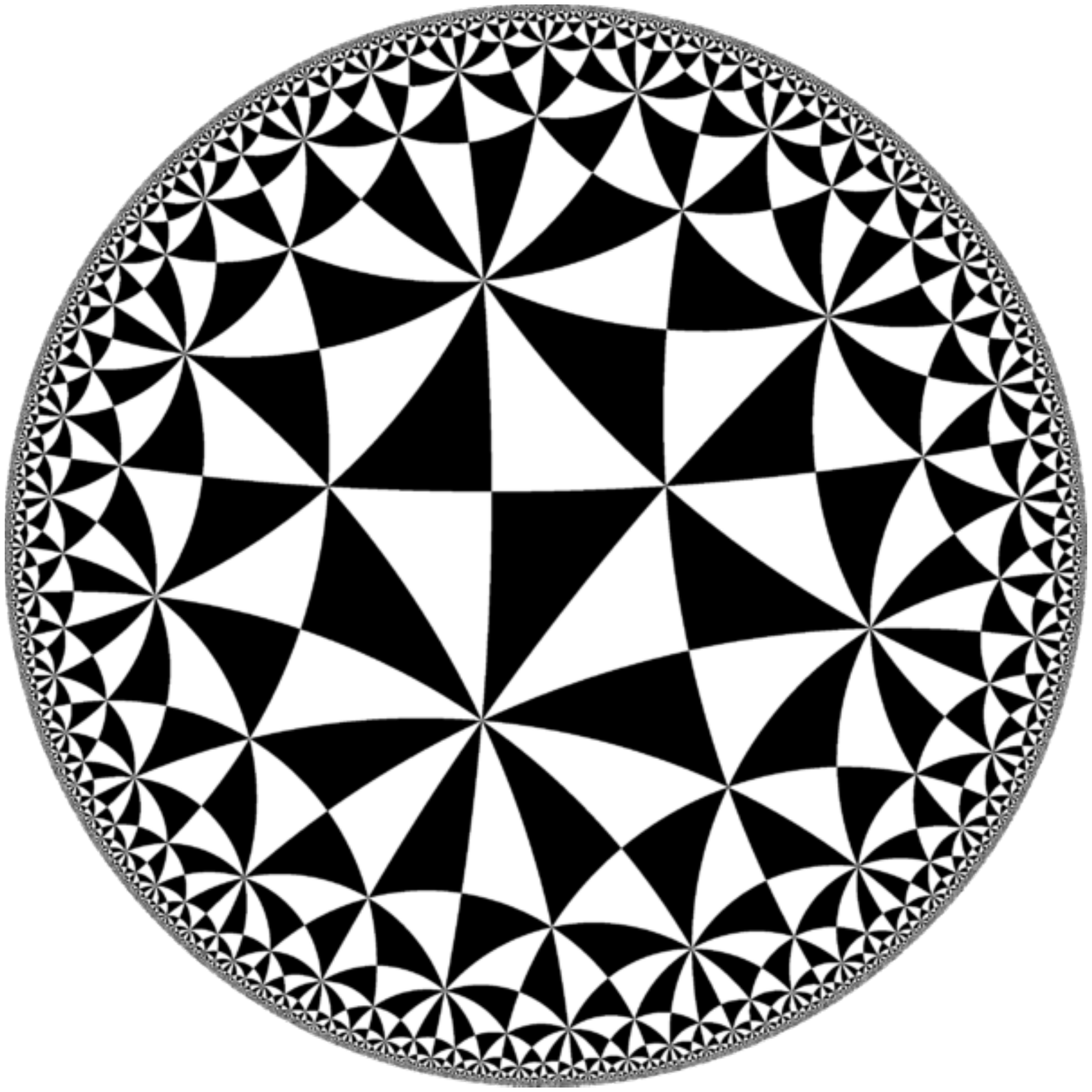}} 
\caption{\label{tiling} Tilings of the Poincar\'e disk \cite{ut}}
\end{figure}

 In our note we are interested in the Poincar\'e disk model.
 So let $\D=\{z\in \C:|z|<1\}$ be  the unit disk in the (complex) plane which we identify with
 $\R^2$. The lines/geodesics with respect to the hyperbolic geometry in this model
  are  arcs of Euclidean circles in $\D$ that
 are orthogonal to the  unit circle $\T:=\{z\in \C:|z|=1\}$ of $\D$
 (see figure \ref{parallel-f}). Given a line  $C$ in the hyperbolic geometry and a  point 
 $a\in \D$ not belonging to $C$,  there are infinitely many hyperbolic lines parallel to $C$
 (in other words disjoint from $C$) and passing through $a$ (see figure \ref{parallel-f}).
 The hyperbolic distance $P(a,b)$  of two points $a$ and $b$ is the hyperbolic length of the
 associated  geodesic and  is therefore given by the 
 integral $L(\gamma):= \int_\gamma \frac{2\;|dz|}{1-|z|^2}$ 
 over the unique circular arc $\gamma$ passing 
 through $a$ and $b$ and orthogonal to $\T$.  Note that $L(\gamma)=\inf L(\Gamma)$,
 where $\Gamma$ is any smooth curve joining $a$ with $b$.
 Or if one prefers a nice formula: 
 $$ \frac{|a-b|^2}{(1-|a|^2)(1-|b|^2)}= \frac{1}{2}\left( \frac{e^{P(a,b)}+e^{-P(a,b)}}{2}-1\right).$$
  
 \begin{figure}[h!]
   \hspace{3.5cm}
   \scalebox{.50} {\includegraphics{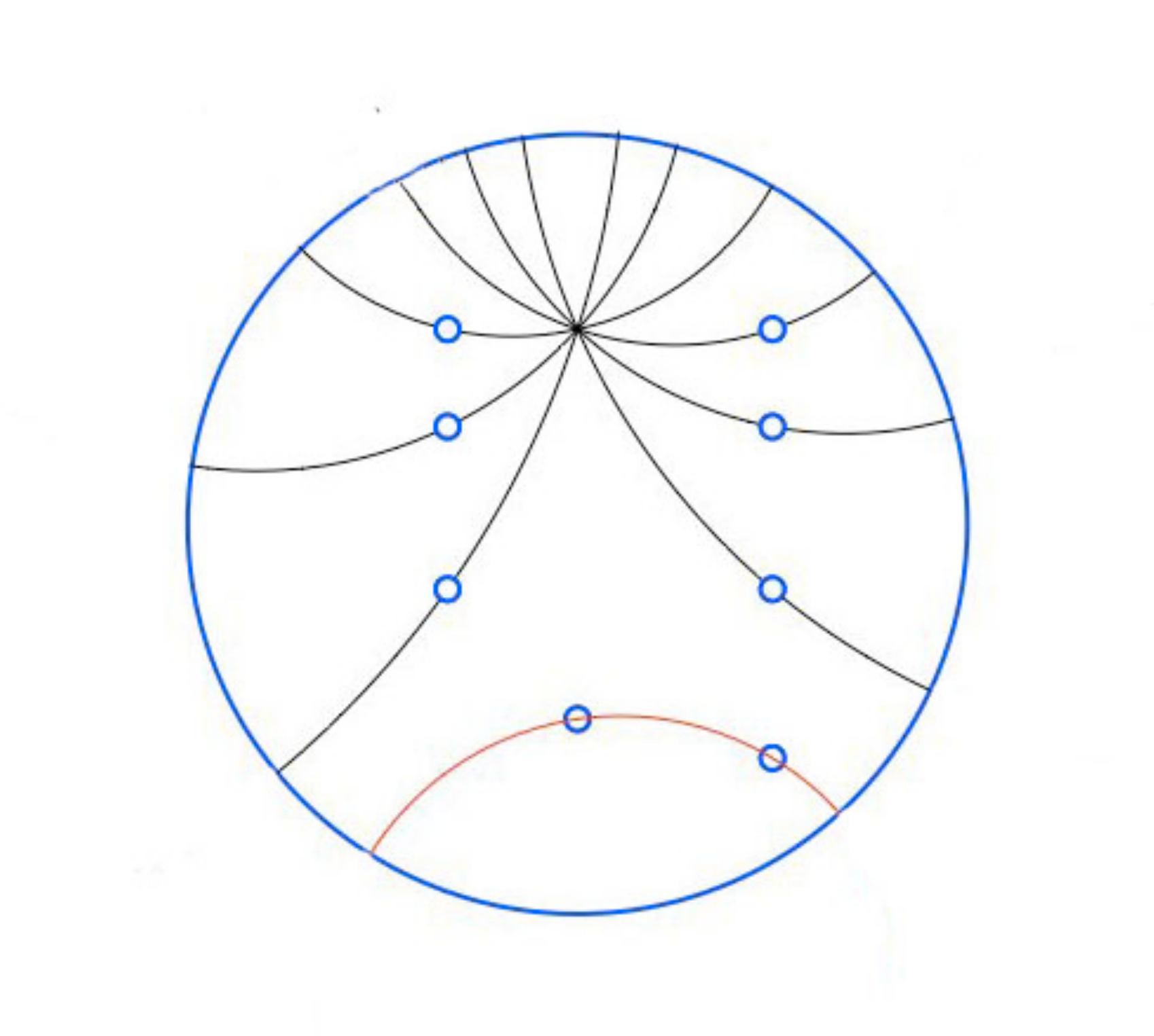}} 
\caption{\label{parallel-f} Infinitely many lines parallel to the red line and 
 passing through the black point }
\end{figure}
  Let
 $$\rho(a,b)=\tanh \left(\frac{1}{2}P(a,b)\right).$$ \noindent Then $\rho(a,b)$ is called the pseudohyperbolic distance  of the two points $a,b$ and is
 given by 
 $$\rho(a,b):=\left|\frac{a-b}{1-\ov a b}\right|.$$
 In other words, 
 $$P(a,b)=\log\frac{1+\rho(a,b)}{1-\rho(a,b)}. $$
 
 It is this pseudohyperbolic distance that we will work with, 
 because this seems to be the most suitable for 
 function theoretic aspects.

 \section{Function theoretic tools}

Given $a\in \D$ and $0<r<1$,  let
$$D_\rho(a,r)=\{z\in \D: \rho(z,a)<r\}$$
be the pseudohyperbolic disk centered at $a$ and with radius $r$.
 It is a simple computational exercise in complex analysis,
 that $D_\rho(a,r)$  coincides with the Euclidean disk $D(p,R)$ where
$$ \mbox{$\dis p=\frac{1-r^2}{1-r^2|a|^2}\; a$ and $\dis R= \frac{1-|a|^2}{1-r^2|a|^2}\; r$}.$$

An  important feature of the hyperbolic metric within function theory comes from
the Schwarz-Pick lemma which  tells us that the holomorphic isometries with respect to $\rho$ (or $P$)
are exactly the conformal self-mappings of the disk: 

\begin{theorem}[Schwarz-Pick Lemma]\label{spl}
Let $f:\D\to\D$ be holomorphic. Then, for every $z,w\in\D$,
$$\rho(f(z),f(w))\leq \rho(z,w),$$
with equality at a pair $(z,w)$, $z\not=w$,  if and only if 
$$f(z)=e^{i\theta}\frac{a-z}{1-\ov a z}$$
for some $a\in \D$ and $\theta\in [0,2\pi[$.
\end{theorem}
\begin{proof}
This is an immediate corollary to the Schwarz lemma  (see e.g. \cite{ru})
 by considering the function 
$$F:=S_{f(w)}\circ f\circ S_w,$$
where for $a\in \D$,
$$S_a(z)= \frac{a-z}{1-\ov a z}$$
is the conformal automorphism of $\D$ interchanging $a$ with the origin.
\end{proof}

Consider now the set of all pseudohyperbolic disks $D_\rho(x,r)$, $x\in\; ]-1,1[$,
with fixed radius $r\in \;]0,1[$. In  studying the boundary behaviour
of holomorphic functions in the disk, 
 it is of interest to know whether the set $\Union_{x\in ]-1,1[} D_\rho(x,r)$ belongs
to a cone 
$$\Delta(\beta):=\Big\{z\in \D: \frac{|{\rm Im}\; z|}{1-{\rm Re}\, z}<\tan\beta\Big\}$$
with cusp at $z=1$ and angle $2\beta$ such that  $0<\beta<\pi/2$. 
A positive answer  is known among specialists in hyperbolic geometry. We never encountered a proof, though, available for function theorists. It is the aim of this note to provide such a proof.
For a nice introduction into the function theoretic aspects of the hyperbolic geometry,
see \cite{bm}.

\section{A union of hyperbolic disks}

\begin{figure}[h!]
   \hspace{0,3cm}
   \scalebox{.50} {\includegraphics{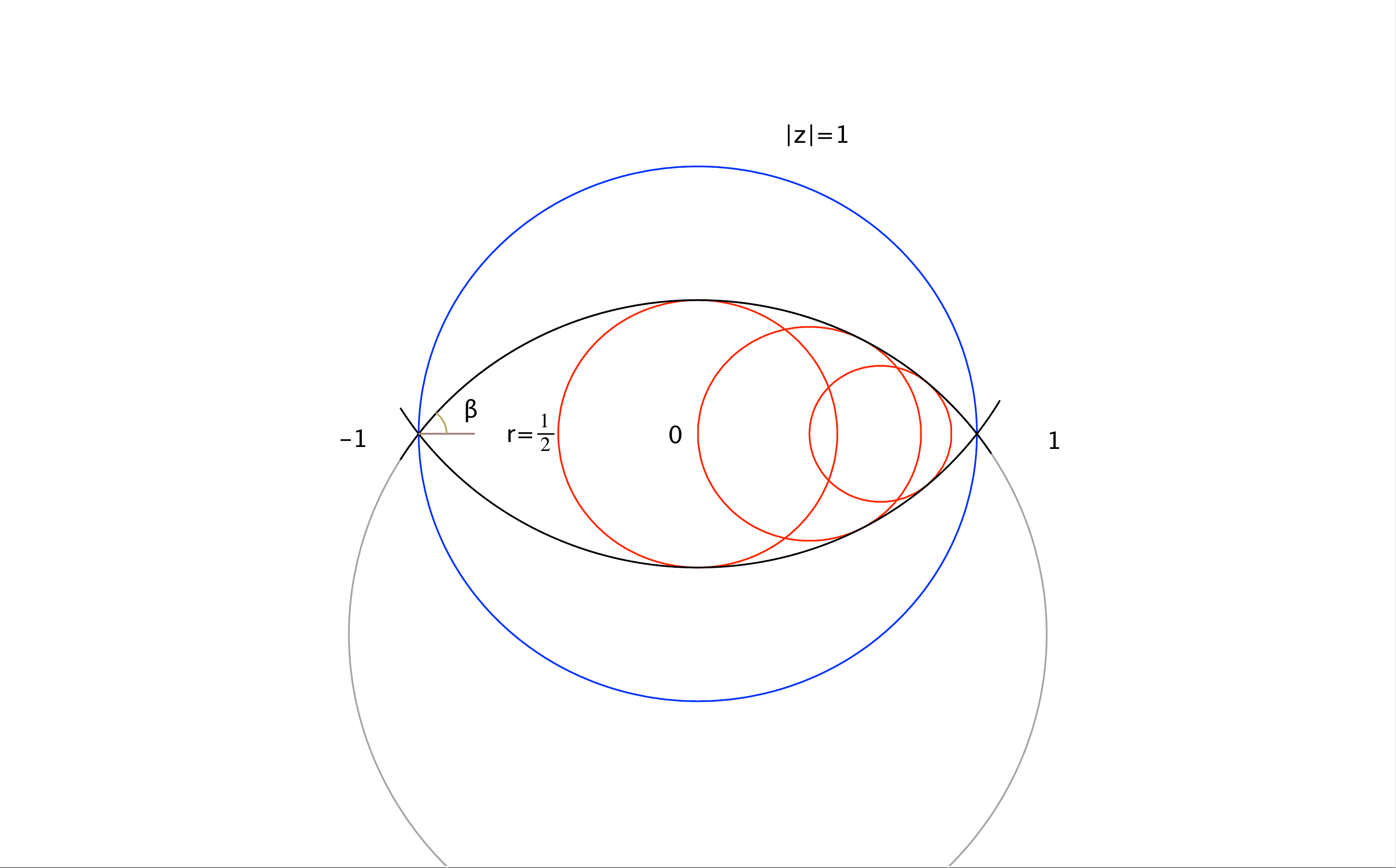}} 
\caption{\label{hyp3} The boundary of a union of hyperbolic disks with fixed radius}
\end{figure}

Here is the assertion we are going to prove.

\begin{theorem}\label{hyperbolic-disks}\hfill
\begin{enumerate}
\item[(1)] The upper boundary  $\mathscr C^+$ of $\dis\Union_{-1<x<1}D_\rho(x,r)$ is 
an arc of the circle
$$\mathfrak C:=\left\{w\in \C: \left|w+i \frac{1-r^2}{2r}\right|=\frac{1+r^2}{2r} \right \},$$
the lower boundary is its reflection with respect to the real axis (see figure \ref{hyp3}).

\item [(2)] The tangens of the angle $\beta$ under  which $\mathscr C^+$ cuts the real axis is 
$2r/(1-r^2)$.
\end{enumerate}
\end{theorem}

 \begin{figure}[h!]
   \hspace{3cm} 
   \scalebox{.40} {\includegraphics{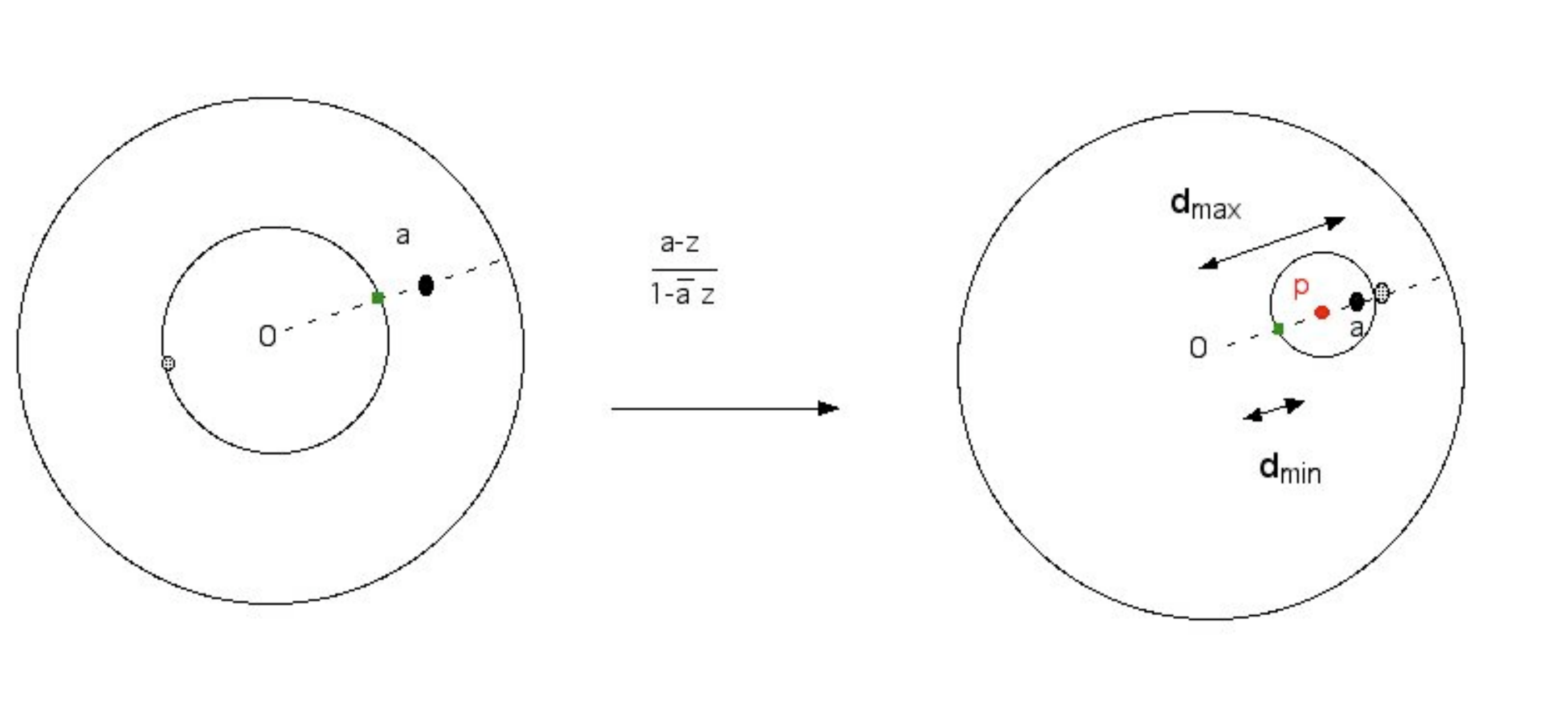}} 
\caption{\label{hyp1} Hyperbolic disks}
\end{figure}
\vspace{3cm}

Before we give our proof,  we  observe  that the largest distance $d_{max}$, respectively, if $|a|>r$, the smallest distance $d_{min}$,  of  a point in $D_\rho(a,r)$ to $0$ are given by
$$\mbox{$\dis d_{min}= \frac{|a|-r}{1-r|a|}$ \; and\; $\dis d_{max}=\frac{|a|+r}{1+r|a|}$}.$$

This can be seen by considering the conformal automorphism of the disk given by
$\varphi(z)=\frac{a-z}{1-\ov a z}$,  by noticing that the image of the disk
$D(0,r)=D_\rho(0,r)$ is the disk $D(p,R)$ and by calculating the images of the boundary points
$\pm r e^{i \arg a}$ which lie on the  half-line passing through $0$ and $a$ (see figure \ref{hyp1}). \medskip

\begin{proof}

(1) The proof  is best done via a conformal mapping  of $\D$ onto the right half-plane (see figure \ref{hyp5}).

\begin{figure}[h]
   \hspace{3cm}
   \scalebox{.30} {\includegraphics{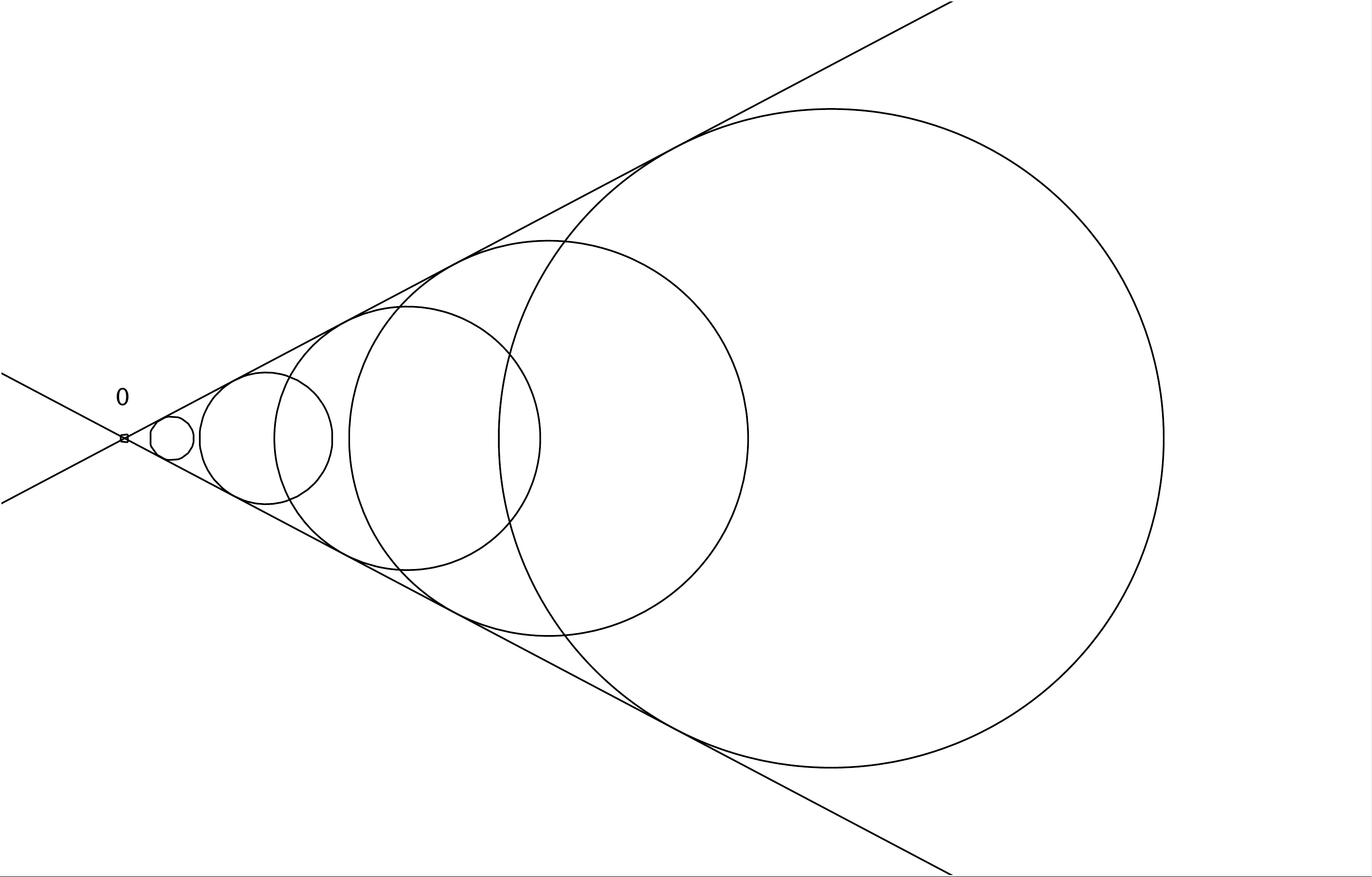}} 
\caption{\label{hyp5} The boundary of a union of hyperbolic disks with fixed radius in the right
half plane}
\end{figure}

Recall that if  $-1<x<1$, then the function $S_x$, given by $S_x(z)=(x-z)/(1-xz)$, maps the disk
$D(0,r)$  onto $D_\rho(x,r)$ with $x_m:=S_x(-r)=(x+r)/(1+xr)$ and 
$x_M:=S_x(r)=(x-r)/(1-xr)$.
Since $S_x$ maps $[-1,1]$ onto $[-1,1]$, and since the circle $D(0,r)$ cuts $[-1,1]$
at a right angle,  the angle invariance property of conformal maps implies that
$D_\rho(x,r)$ is the disk passing through the points $x_m$ and $x_M$ and orthogonal to $[-1,1]$.
Now we switch to the right half-plane $H$ by using  the map $\Psi(z):=(1+z)/(1-z)$ of $\D$ onto $H$. Then, by a similar reasoning, $K:=\Psi(D_\rho(x,r))$ is the disk orthogonal to the real axis and passing through
the points 
$$\mbox{$w_m:=\dis\Psi(x_m)= \frac{1-r}{1+r}\;\frac{1+x}{1-x}$ \;and\; $w_M:=\dis\Psi(x_M)= 
 \frac{1+r}{1-r}\;\frac{1+x}{1-x}$.}$$
 Hence the center $C_x$ of $K$ is the arithmetic mean $(w_M+ w_m)/2$ of $w_m$ and $w_M$,  
 and the radius $R_x$  is $(w_M- w_m)/2$. Thus
  $$\mbox{$\dis C_x=\frac{1+r^2}{1-r^2}\;\frac{1+x}{1-x}$ and 
  $\dis R_x=\frac{2r}{1-r^2}\;\frac{1+x}{1-x}$}.$$ 
Note that if the center $x$ of the pseudohyperbolic disk $D_\rho(x,r)$ runs through $]-1,1[$, then the center $C_x$ of the  Euclidean disk $\Psi(D_\rho(x,r))$ runs through $]0,\infty[$.
Due to conformal invariance,  the boundary $\mathscr C$ of $\Union_{-1<x<1}D_\rho(x,r)$ 
coincides with the preimage $\Psi^{-1}( \widetilde{\mathscr C})$ of the boundary 
$ \widetilde{\mathscr C}$ of 
$$S:=\Union_{-1<x<1} \Psi(D_\rho(x,r))=\Union_{-1<x<1}K(C_x,R_x).$$
We show that  $\widetilde{\mathscr C}$ is the union of the half-line $\{e^{i\beta}t: t\geq 0\}$
and its reflection $\{e^{-i\beta}t: t\geq 0\}$,  where $\beta$ is the angle whith  $\sin \beta= 2r/(1+r^2)$
(for a first glimpse, see figure \ref{hyp5}).

For a fixed $x\in\;]-1,1[$, consider in the first quadrant  the tangent $T$ 
to $K(C_x,R_x)$ that passes through the origin. Let
$\beta_x$ be its angle with respect to the real axis.  Then 
$$\sin\beta_x=\frac{R_x}{C_x}=\frac{2r}{1+r^2}.$$
This is independent of $x$. Hence $T$ is  a joint tangent to all the Euclidean disks $K(C_x,R_x)$.
In other words,  $S$ is contained in the infinite triangle $\Delta$ formed by $T$ and its reflection.  
To show that $\Delta=S$,  we need to prove that every  point  on $T$ is the tangent point of
some of the disks $K(C_a,R_a)$ with $-1<a<1$. 
To this end, let $P$ be the  point  on $T$ whose distance to $0$ is $t$ and let $T_t$ be the line orthogonal to $T$ and passing
through $P$. Then $T_t$ cuts the real line at a point $x_t$. The unique disk $K$  centered at $x_t$ 
and having $P$ as its tangent point to $T$ has center $C(t)$ and radius $R(t)$, which are given by 
(see figure \ref{hyp6})
$$\mbox{$\dis C(t)=x_t$ and $\dis R(t)=x_t \sin \beta=x_t \;2r/(1+r^2)$}.$$

\begin{figure}[h]
   \hspace{2cm}
   \scalebox{.50} {\includegraphics{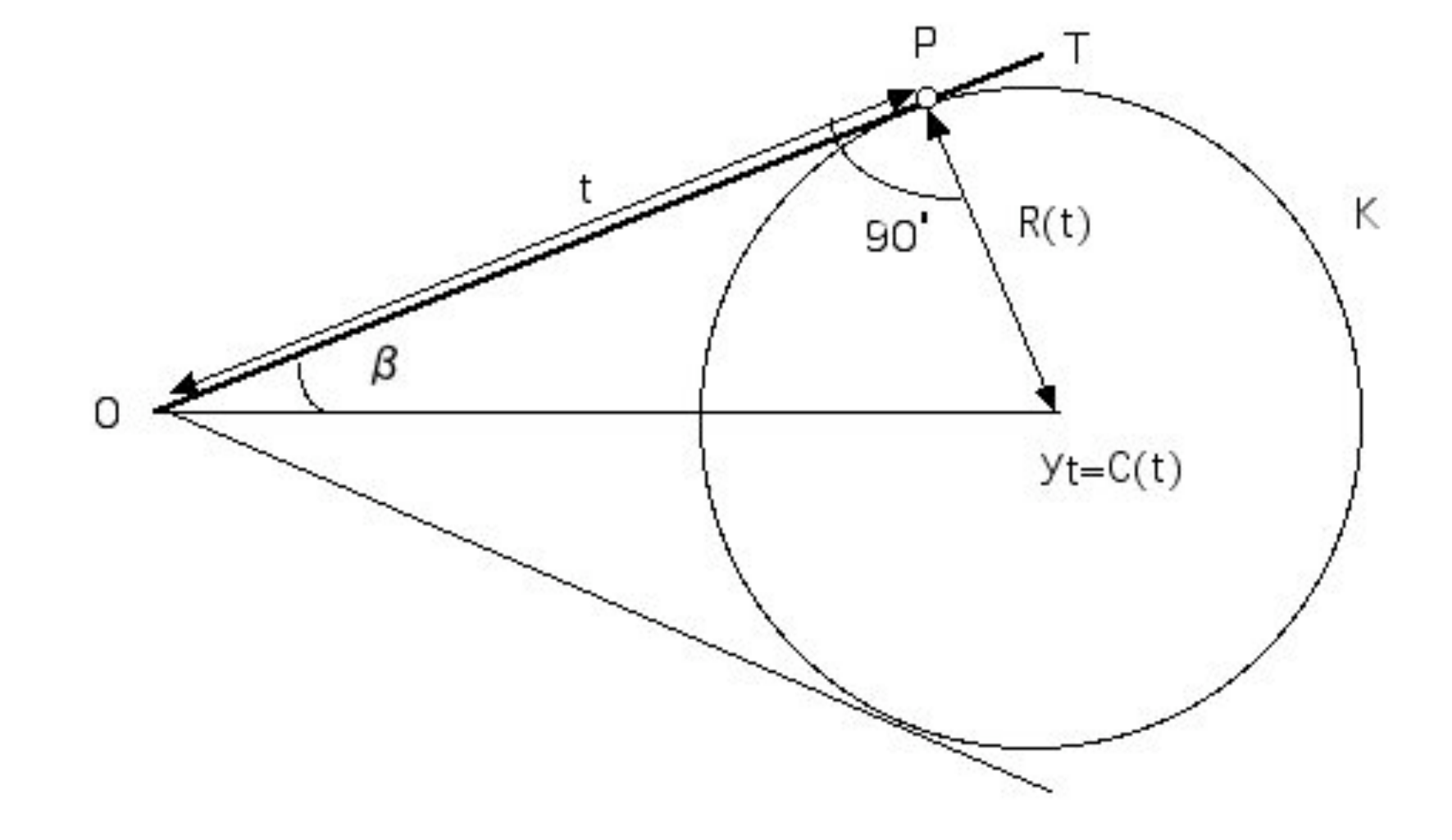}} 
\caption{\label{hyp6} The tangent $T$}
\end{figure}

Now 
$$x_t=\frac{t}{\cos\beta}=\frac{t}{\sqrt{1-\sin^2\beta}}=t\;\frac{1+r^2}{1-r^2}.$$
But $t=(1+a)(1-a)$ for a unique $a\in \;]-1,1[$. Thus 
$$x_t=\frac{1+a}{1-a}\frac{1+r^2}{1-r^2}=C_a$$ and 
$$R(t)=x_t \;2r/(1+r^2)= \frac{1+a}{1-a} \frac{2r}{1-r^2}=R_a.$$
We conclude that $$K=\Psi( D_\rho(a, r))=K(C_a,R_a).$$

 \begin{figure}[h!]
   \hspace{3cm}
   \scalebox{.50} {\includegraphics{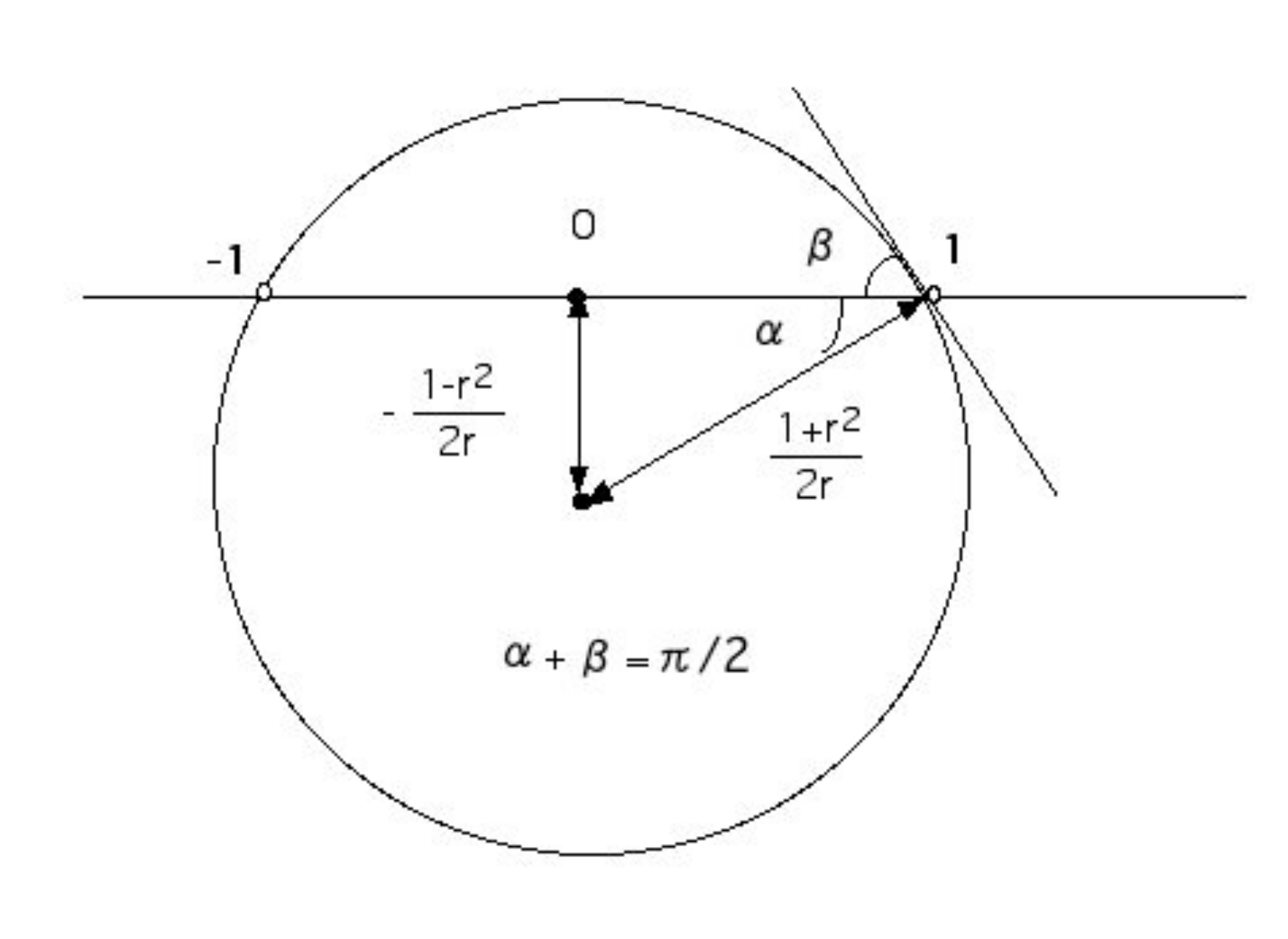}} 
\caption{\label{hyp4} The angle $\beta$}
\end{figure}

By moving back from the right half-plane to the unit disk, we see that $\mathscr C^+=\Psi^{-1}(T)$
is an arc of  a circle $\mathfrak C$ which passes  through $-1$ and $1$ and cuts twice  the axis $[-1,1]$ under the
 the angle $\beta$ with  $\sin\beta= 2r/(1+r^2)$.  Using figure \ref{hyp4}, we
 then deduce that the radius $R$ of $\mathfrak C$ coincides with the hipotenusa of the 
 displayed triangle and so  
 $$R= \frac{1}{\cos \alpha}=\frac{1}{\sin(\frac{\pi}{2}-\alpha)}=
\frac{1}{\sin\beta}=\frac{1+r^2}{2r}.$$
 This implies that the center $C$ of $\mathfrak C$ is given by
 $$C= -iR \sin \alpha=  -i\frac{1+r^2}{2r} \sqrt{1-\cos^2 \alpha}= -i\frac{1+r^2}{2r}\; 
 \frac{1-r^2}{1+r^2}=- i\frac{1-r^2}{2r}.$$

(2) $\tan\beta= \sin (\pi/2-\alpha)/ \cos(\pi/2-\alpha)=\cos\alpha/ \sin\alpha$ with 
$$\mbox{$\dis \sin\alpha=\frac{\frac{1-r^2}{2r}}{\frac{1+r^2}{2r}}= \frac{1-r^2}{1+r^2}$
and $ \cos \alpha= \dis\frac{1}{\frac{1+r^2}{2r}}=\frac{2r}{1+r^2}$}.$$
Hence  $\tan\beta= 2r/(1-r^2)$ (see figure \ref{hyp4}).

 \end{proof}

A purely computational proof can be found in \cite[Appendix]{noel}. There it is also shown
that the  Euclidean length of $\mathscr C^+$ is $2\frac{1+r^2}{r} \arctan r$, and that the surface
enclosed by $\Union_{-1<x<1} D_\rho(x,r)$ has Euclidean measure
$\left( \frac{1+r^2}{r}\right)^2 \arctan r - \frac{1-r^2}{r}$.

\vspace{-0,6cm}

 \end{document}